\documentclass[11pt]{amsart}

\usepackage{enumitem}
\usepackage{amssymb, amsmath}
\usepackage{amsthm}
\usepackage{amsrefs}
\usepackage{hyperref}
\usepackage{cleveref}

\usepackage{cases}

\usepackage{tikz}
\usepackage{tikz-cd}
\usetikzlibrary{decorations.markings, decorations.pathreplacing, positioning,arrows, matrix, decorations.pathmorphing, shapes.geometric, calc}
\usetikzlibrary{backgrounds, decorations}

\newcounter{alancomments}

\newcounter{alexcomments}

\newtheorem{theorem}{Theorem}[section]
\newtheorem{lemma}[theorem]{Lemma}

\newtheorem{remark}[theorem]{Remark}
\newtheorem{proposition}[theorem]{Proposition}
\newtheorem{thmletter}{Theorem}

\newcommand{\p}[1]{\noindent {\newline\bf #1.}}

\newcommand{\PCP}{\operatorname{PCP}}
\newcommand{\NPCP}{\operatorname{NPCP}}

\newcommand{\eq}{\operatorname{Eq}}
\newcommand{\im}{\operatorname{Image}}

\title[PCP for certain groups]{Post's Correspondence Problem for Hyperbolic and Virtually Nilpotent Groups}

\author{Laura Ciobanu}
\address{ School of Mathematical and Computer Sciences,
 Heriot--Watt University, and Maxwell Institute for Mathematical Sciences, Edinburgh EH14 4AS, Scotland}
 \email{l.ciobanu@hw.ac.uk}

\author{Alex Levine}

\address{Department of Mathematics, The University
of Manchester, Manchester M13 9PL, UK}

\email{alex.levine@manchester.ac.uk}

\author{Alan D. Logan}
\address{School of Engineering and Physical Sciences,
 Heriot--Watt University, Edinburgh EH14 4AS, Scotland}
 \email{Alan.Logan@hw.ac.uk}



\begin{document}

\begin{abstract}
Post's Correspondence Problem (the PCP) is a classical decision problem in theoretical computer science that asks whether for pairs of
free monoid morphisms $g, h\colon\Sigma^*\to\Delta^*$ there exists any non-trivial $x\in\Sigma^*$ such that $g(x)=h(x)$.

Post's Correspondence Problem for a group $\Gamma$ takes pairs of group homomorphisms $g, h\colon F(\Sigma)\to \Gamma$ instead, and similarly asks whether there exists an $x$ such that $g(x)=h(x)$ holds for non-elementary reasons. The restrictions imposed on $x$ in order to get non-elementary solutions lead to several interpretations of the problem; we mainly consider the natural restriction asking that $x \notin \ker(g) \cap \ker(h)$ and prove that the resulting interpretation of the PCP
is undecidable for arbitrary hyperbolic
$\Gamma$, but decidable when $\Gamma$ is virtually nilpotent, en route also studying this problem for finite extensions.

We also consider a different interpretation of the PCP due to Myasnikov, Nikolaev and Ushakov, proving decidability for torsion-free nilpotent groups.
\end{abstract}
\maketitle

\section{Introduction}
  Post's Correspondence Problem (the PCP) is a prominent undecidable problem in Computer
  Science. It takes as input a pair of free monoid morphisms $g,
  h\colon\Sigma^*\to\Delta^*$, and asks if there exists any non-trivial $x\in\Sigma^*$ such that
  $g(x)=h(x)$. Undecidability was proven by Post, from whom it takes its name
  \cite{Post1946Correspondence}, and we refer the reader to the survey of Harju
  and Karhum\"{a}ki for background and applications \cite{Harju1997Morphisms}
  (see also the recent article of Neary \cite{Neary2015undecidability}). The
  prominence of the PCP is due to its role as a simple source of
  undecidability: for matrix (semi)group decision problems, tiling problems, questions about context-free grammars, and in many other contexts.

  In this paper we consider the PCP for groups, and define it as follows. Let $\Sigma$ be a finite alphabet and
  $F(\Sigma)$ the associated free group, $\Gamma$ a group, and $g, h\colon
  F(\Sigma)\rightarrow \Gamma$ group homomorphisms. The \emph{equaliser} of
  $g, h$ is the subgroup $\eq(g, h)=\{x\in F(\Sigma) \mid g(x)=h(x)\}$ of $F(\Sigma)$. \emph{Post's Correspondence Problem for} $\Gamma$
asks whether there exists $x\in \eq(g, h)\setminus(\ker(g)\cap\ker(h))$. An equivalent formulation of the PCP and the reason why we remove $\ker(g)\cap\ker(h)$ from the equaliser are given in Section \ref{sec:prelim}. We study the PCP for entire classes of groups $\Gamma$, and prove the following theorems.

  \begin{thmletter}
  \label{thm:hyperbolic}
  The PCP for hyperbolic groups is undecidable.
  \end{thmletter}
The proof uses the undecidability of the subgroup membership problem, via a version of Rips' construction.
  Our full result is stronger than we state here, proving undecidability of the ``binary'' PCP for torsion-free hyperbolic groups; see Theorem \ref{thm:hyperbolicBODY}.

One motivation to study the PCP in hyperbolic groups comes from the fact that the PCP for free groups can be traced to Stallings in the 1980s (see \cite{Stallings1987Graphical}), and while recent results settle the PCP for certain classes of free group maps  \cite{Bogopolski2016algorithm, CiobanuLoganJAlg, CiobanuLoganDLT, ciobanu_et_al:LIPIcs:2020:12527, JPM}, its solubility for general free group maps remains an important open question (\cite[Problem 5.1.4]{Dagstuhl2019}).

Secondly, we have a decidability result for virtually nilpotent groups.

  \begin{thmletter}
  \label{thm:nilpotent}
  The PCP is decidable for finitely generated virtually nilpotent groups.
  \end{thmletter}

To prove Theorem \ref{thm:nilpotent} we consider the \emph{non-homogeneous} PCP
(NPCP), which is analogous to the generalised PCP (GPCP) in free monoids.
Here, we prove that for a group $K$, the PCP
for virtually $K$ groups reduces to the NPCP and PCP for $K$ (Proposition
\ref{thm:virtually}).  This proof is similar to the proof that the conjugacy
problem for virtually $K$ groups reduces to the \emph{twisted conjugacy
problem} for $K$.  Indeed, the twisted conjugacy problem itself reduces to the
NPCP, suggesting a connection between the PCP and the conjugacy problem.

\p{Turing reduction}
Many of our results have the form: the decidability of a certain algorithmic problem, $\mathcal{P}$ say, implies the decidability of a different algorithmic problem, $\mathcal{Q}$ say.
In each case, we actually prove that $\mathcal{Q}$ Turing reduces to $\mathcal{P}$, i.e. that we can use an oracle for $\mathcal{P}$ to solve $\mathcal{Q}$.
Lemma \ref{lem:NPCP} further proves that two problems are Turing equivalent.
Although the language of Turing reduction is standard in complexity theory,
we chose the more informal (but still correct) language on implications of decidability in order to be more accessible to a wider audience.

  \p{Outline of the article}
  In Section \ref{sec:prelim} we introduce the version of the PCP we are working with.
  In Section \ref{sec:hyperbolic} we prove Theorem \ref{thm:hyperbolic}.
  In Section \ref{sec:virtually} we connect the NPCP for a group $K$ to the PCP for virtually $K$ groups.
  In Section \ref{sec:nilpotent} we consider the NPCP for $\mathfrak{A}$ groups and the PCP for virtually $\mathfrak{A}$ groups, where $\mathfrak{A}$ is the
  class of recursively presented groups in a variety of groups.
  These considerations allow us to prove Theorem \ref{thm:nilpotent}.
  
  In Section \ref{sec:verbal} we consider a different interpretation of the PCP for groups, which we call the \emph{verbal} PCP and is due to Myasnikov, Nikolaev and Ushakov, and prove decidability of this problem for torsion-free nilpotent groups.

  \section*{Acknowledgements}
  The first and third named authors were partially supported by EPSRC Standard Grant EP/R035814/1. The first named author would like to thank Bettina Eick for helpful discussions. The second named author acknowledges the support of the London Mathematical Society and the Heilbronn Institute for Mathematical Research. The second and third named authors are also grateful to the University of St Andrews, who hosted them while completing this work. All authors thank the anonymous referee for many useful comments and suggestions.

  \section{Post's Correspondence Problem for Groups}
  \label{sec:prelim}

We will often refer to the PCP for free monoids as the `classical PCP'.

  Here we define our main version of PCP for groups, formally expanding on the initial definition given in the Introduction.
A \emph{recursive presentation} is a presentation with a finite generating set and a recursive (e.g. finite) set of relators.
A group is \emph{recursively presented} if it admits a recursive presentation; recursively presented groups form the largest class of groups for which standard notions of computation are defined, and hence we restrict our definition and studies to this class.

  An \emph{instance of the PCP} is a four-tuple $I=(\Sigma, \Gamma, g, h)$,
  where $\Sigma$ is a finite alphabet and
  $F(\Sigma)$ is the associated free group, $\Gamma$ is a group given by a recursive (e.g. finite) presentation, and $g, h\colon
  F(\Sigma)\rightarrow \Gamma$ are group homomorphisms. The \emph{equaliser} of
  $g, h$ is the subgroup $\eq(g, h)=\{x\in F(\Sigma) \mid g(x)=h(x)\}$ of $F(\Sigma)$.

\p{The (kernel-based) PCP}  \emph{Post's Correspondence Problem for groups} itself, hereafter \emph{the PCP} or \emph{kernel-based PCP},
  is the decision problem:

  \noindent{\newline
  Given $I=(\Sigma, \Gamma, g, h)$, is the group $\eq(g, h)/(\ker(g)\cap\ker(h))$ trivial?}
  \newline

  By a \emph{solution} to $I$ we mean an element $x\in \eq(g, h)\setminus(\ker(g)\cap\ker(h))$. Solutions are therefore those elements $x\in F(\Sigma)$ that correspond to non-trivial cosets $x(\ker(g)\cap\ker(h))\in\eq(g, h)/(\ker(g)\cap\ker(h))$.

 If at least one of $g$ and $h$ is injective we get $\ker(g)\cap\ker(h)=\{1\}$, so the PCP has the same statement as the classical one (``is there any $x \neq 1$ such that $g(x)=h(x)$?''). However, when neither map is injective, $\ker(g)\cap\ker(h)$ is {always} non-trivial as either $F(\Sigma)$ is cyclic, and here all non-trivial subgroups have non-trivial intersection, or $F(\Sigma)$ is non-cyclic whence $\ker(g)\cap\ker(h)$ contains the non-trivial subgroup $[\ker(g), \ker(h)]$ (see \cite[Lemma 1]{CiobanuLoganDLT}), and so while the non-triviality of $\eq(g, h)$ is established, it does not capture the core of the problem. We therefore quotient out by $\ker(g)\cap\ker(h)$ as we wish to consider the case when neither map is injective.

How we deal with non-injective maps is important,
  as considering the classical PCP, this is where we expect the
  undecidability to lie: The standard undecidability proofs
  are based on pairs of non-injective maps, while all known proofs of undecidability of the PCP
  for pairs of injective maps proceed via reversible Turing machines and are in comparison extremely technical \cite{Lecerf1963, Ruohonen1985Reversible, saarela2010noneffective}.
We indeed use pairs of non-injective maps in our undecidability proof, \Cref{thm:hyperbolic} (non-injectivity is explained after the proof).

\p{Properties of the PCP}
We now note a couple of interesting properties of groups with decidable PCP.
Firstly, groups with decidable PCP have decidable word problem.

\begin{lemma}\label{lem:WP}
If a recursively presented group has decidable (kernel-based) PCP, then it has decidable word problem.
\end{lemma}
\begin{proof}
Let $\Gamma$ have recursive presentation $\langle X \mid R \rangle$, and let $w$ be a word over $X$.
Consider the PCP instance $I=(\{a\}, \Gamma, g, g)$, so take as the domain free group the infinite cyclic group $\mathbb{Z}$, generated by $a$, and let $g=h$ with $g(a)=w$.
Then $\eq(g,g)=\langle a \rangle$, and if there exists a solution to this PCP instance, then $\eq(g,g)=\langle a \rangle \neq \ker(g)$, so $w\neq 1$ in $\Gamma$, while if no solution exists then $\eq(g,g)=\langle a \rangle = \ker(g)$, so $w=1$ in $\Gamma$.\end{proof}

Next, the decidability of PCP is preserved under taking subgroups.
\begin{lemma}
	\label{lem:subgroups}
	Suppose $H$ is a finitely generated subgroup of a recursively presented group $\Gamma$.
	If the PCP is decidable for $\Gamma$, then it is decidable for $H$.
\end{lemma}

\begin{proof}
By Higman's Embedding Theorem, $H$ is recursively presented, so the PCP for $H$ makes sense.

Let $I = (\Sigma, H, g, h)$ be an instance of the
PCP for $H$. Since $H \leq \Gamma$, we can extend the codomains of \(g\) and \(h\) to obtain homomorphisms \(\bar{g}, \ \bar{h} \colon
F(\Sigma) \to \Gamma\), with $h(x)=\bar{h}(x)$ and $g(x)=\bar{g}(x)$ for all $x\in F(\Sigma)$. Moreover, \(\eq(g, h) =
\eq(\bar{g}, \bar{h})\), \(\ker(g) = \ker(\bar{g})\) and \(\ker(h) =
\ker(\bar{h})\). Thus \(I\) admits a solution if and only if the instance
\(\bar{I} = (\Sigma, \Gamma, \bar{g}, \bar{h})\) of the
PCP for $\Gamma$ admits a solution. The result now follows immediately.
\end{proof}
\p{The verbal PCP}
  Myasnikov, Nikolaev and Ushakov have previously defined and studied a version of the PCP for groups \cite{Myasnikov2014Post}, which we call the \emph{verbal PCP} and describe at length in Section \ref{sec:verbal}.
  In particular, they did not prove undecidability for any classes of groups.

  The key difference between our interpretation and theirs is how pairs of non-injective maps are dealt with: the verbal PCP mitigates against $\ker(g)\cap\ker(h)$ automatically being non-trivial by considering varieties of groups.
  However, this mitigation is not robust enough to deal with $\Gamma$ being non-elementary hyperbolic, or more generally not being contained in any proper variety of groups.
 We discuss the differences further in Section \ref{sec:verbal}.

\section{Hyperbolic groups}
\label{sec:hyperbolic}

Hyperbolic groups are possibly the most studied class of infinite discrete groups in the last few decades. In particular, there has been significant work on their algorithmic properties, and both decidability and complexity results have been obtained by exploiting their geometry and combinatorics (thin triangles, regular geodesics, etc). For an introduction to hyperbolic groups, \cite[Chapter 6]{groups_langs_aut} offers an account suitable for algorithmic purposes.
The most fundamental results here are that in any hyperbolic group
the word and conjugacy problem are solvable in linear time \cite{CGW}.

More recently, important work on algorithms in hyperbolic groups has been inspired by developments in
Computer Science. For example, work of Plandowski, Je\.z, Diekert and others on
\textsf{PSPACE} algorithms to solve equations in free monoids and groups using
compression \cite{DiekertJezPlandowski2016Compression} has been applied to prove
that the compressed word problem and compressed simultaneous conjugacy problem
in hyperbolic groups are solvable in polynomial time \cite{Holt2019Compressed}. Compression techniques, together with the decidability of systems of equations by Dahmani and Guirardel \cite{Dahmani2010Foliations}, also led to
the characterisation of solution sets to such systems from a language-theoretic point of view \cite{Ciobanu2019Solutions,
Ciobanu2021Complexity}. As another example, the knapsack problem is a
fundamental algorithmic problem in Computer Science, and Lohrey has recently
proven that the analogous problem for hyperbolic groups is
\textsf{LOGSPACE}-reducible to a context-free language
\cite{Lohrey2020Knapsack}.

However, there are exceptions to the decidability results mentioned above.
For example, there is no algorithm to compute finite generating sets for intersections of finitely generated subgroups \cite{Rips1982subgroups}.
More pertinent to this paper is the \emph{membership problem for subgroups}, which for a fixed recursively presented group $\Gamma=\langle X\mid R\rangle$ and finitely generated subgroup $A=\langle a_1, \ldots, a_n\rangle$ of $\Gamma$, is the problem of determining if a word $w\in F(X)$ defines an element of $A$; the word problem is the special case of $A=\{1\}$.
We say $\Gamma$ has \emph{undecidable subgroup membership problem} if it has a finitely generated subgroup with undecidable membership problem.
Again, this problem is undecidable in general for hyperbolic groups \cite{Rips1982subgroups}.
The undecidability of the subgroup membership problem will be used in Theorem \ref{thm:hyperbolicBODY} below.

\p{Rips' construction}
Our proof on the PCP for hyperbolic groups is based on a version of Rips' construction due to Belegradek and Osin \cite{Belegradek2008Rips}.
The classical Rips' construction takes as input a finitely presented group $Q$ and constructs a short exact sequence $1\to N\to \Gamma\to Q\to1$ where $\Gamma$ is hyperbolic and $N$ is finitely generated; properties of $Q$ impact $N$, in particular if $Q$ has undecidable word problem then the subgroup $N$ has undecidable membership problem.
Belegradek and Osin improved this as follows.

\begin{theorem}[\cite{Belegradek2008Rips}]
\label{thm:BO_Rips}
For every finitely presented group $Q$ and hyperbolic group $H$ there exists a short exact sequence $1\to N\to \Gamma\to Q\to 1$ such that $\Gamma$ is hyperbolic and $N$ is a homomorphic image of $H$.
\end{theorem}

We can view this as the diagram in \Cref{figure:Rips}.

\begin{figure}[h]
\begin{center}
    \begin{tikzcd}
  & H \dar[two heads]{} && \\
  1 \rar{} &N \rar[tail]{} & \Gamma \rar[two heads]{} & Q \rar{}& 1
  \end{tikzcd}
\end{center}
\caption{The Belegradek and Osin version of Rips' construction.}
\label{figure:Rips}
\end{figure}

This is stronger than the classical setting as now properties of both $Q$ and $H$ impact $N$; in particular, if $H$ has trivial abelianisation then so does $N$.
We now construct the seed group $H$ used in the proof of Theorem \ref{thm:hyperbolicBODY}.

\begin{lemma}
\label{lem:seed_group}
There exists a $2$-generated torsion-free hyperbolic group with trivial abelianisation.
\end{lemma}

\begin{proof}
As remarked by Kapovich and Wise \cite[page 2]{Kapovich2000equivalence}, the presentation
\begin{align*}
\langle a, b\mid
a&=[a, b][a^2, b^2]\cdots[a^{100}, b^{100}],\\
b&=[b, a][b^2, a^2]\cdots[b^{100}, a^{100}]\rangle
\end{align*}
is $C'(1/6)$ with relators of the form $a=w$ and $b=v$ with $w, v\in[F(a, b), F(a, b)]$, so defines a hyperbolic group with trivial abelianisation.
As the presentation is small cancellation and neither of the two relators are proper powers, the group is torsion-free \cite[Theorems 3 \& 4]{Huebschmann1979cohomology}.
\end{proof}

\p{The PCP for hyperbolic groups}
  Placing restrictions on the alphabet $\Sigma$, group $\Gamma$ and maps $g$ and $h$ allows us to investigate the boundary between decidability and undecidability.
  The \emph{binary PCP} is the PCP restricted to those instances $I=(\Sigma, \Gamma, g, h)$ where $|\Sigma|=2$.
  For $\mathfrak{X}$ a class of recursively presented groups, the \emph{PCP for $\mathfrak{X}$} is the PCP restricted to those instances $I=(\Sigma, \Gamma, g, h)$ where the group $\Gamma$ is in $\mathfrak{X}$.
  We can intersect such classes of instances, and so for example can consider the binary PCP for torsion-free hyperbolic groups.

  It is interesting to contrast this result to the classical binary PCP (for free monoids), which is decidable \cite{Ehrenfeucht1982generalized}. The binary PCP for free groups remains open.

  \begin{theorem}
  \label{thm:hyperbolicBODY}
  The binary PCP for torsion-free hyperbolic groups is undecidable.
  \end{theorem}

\begin{proof}
Take $H$ in Theorem \ref{thm:BO_Rips} (i.e. in Belegradek--Osin's Rips' construction) to be a $2$-generated torsion-free hyperbolic group with trivial abelianisation, which we know exists by Lemma \ref{lem:seed_group}.
As $N$ is a homomorphic image of $H$, this group is also $2$-generated with trivial abelianisation.
Now, take $Q$ to be torsion-free with undecidable word problem (such a group exists, see for example \cite{Collins1999aspherical}).
Then $\Gamma$ is torsion-free \cite[Theorem 1.1.d]{Belegradek2008Rips}.

For any element $y\in \Gamma \setminus \{1\}$, we define the instance $I_y=(\{a, b\}, \Gamma, g, h)$ of the binary PCP for $\Gamma$ by setting $g, h\colon F(a, b)\to \Gamma$ to be the maps defined by $g(a)=1, g(b)=y$ and $h(a), h(b)$ to be generators for $N \leq \Gamma$; that is, $\langle h(a), h(b)\rangle =N$.
We will now show that $y \in N$ if and only if $\eq(g, h) / (\ker(g) \cap \ker(h))$ is non-trivial.

Suppose first that $y\not\in N$. As $Q$ is torsion-free, we have that $y^n\not\in N$ for all $n$.
Therefore, $\im(g)\cap \im(h) = \langle y \rangle \cap N=\{1\}$ is trivial, so $\eq(g, h)/(\ker(g)\cap \ker(h))$
is trivial.

Suppose, for the other implication, that $y\in N$; we will prove that $\eq(g, h)/(\ker(g)\cap \ker(h))$ is non-trivial.
Since $h(b)\in N$ and $y\in N$, as $N$ has trivial abelianisation, we get $y(h(b))^{-1}\in[N, N]$.
Therefore, there exists some word $U\in [F(a, b), F(a, b)]$ such that $y(h(b))^{-1}=h(U)$.
Moreover, $U\in [F(a, b), F(a, b)]\leq\ker(g)$, so the identity just obtained gives us $h(Ub)=y=g(U)y=g(Ub)$ and
thus $Ub\in\eq(g, h)$; since $Ub\in\eq(g, h)$ and $Ub \notin \ker(g)$, we get that $\eq(g, h)/(\ker(g)\cap \ker(h))$ is non-trivial, as claimed.

Therefore, $\eq(g, h)/(\ker(g)\cap \ker(h))$ is trivial if and only if $y\not\in N$. As $Q$ has undecidable word problem, $N$ has undecidable membership problem.
The result follows.
\end{proof}

Theorem \ref{thm:hyperbolic} follows immediately from Theorem \ref{thm:hyperbolicBODY}, as for example the binary PCP is a special case of the PCP.

Note that neither of the maps $g, h\colon F(a, b)\to\Gamma$ in the proof of Theorem \ref{thm:hyperbolicBODY} are injective.
For $g$, this is because the image is cyclic so clearly not free of rank $2$.
For $h$, this is because the image is the group $N$, which has trivial abelianisation so again not free of rank $2$.


\section{Finite index overgroups and the NPCP}
\label{sec:virtually}

Let $K$ be a finitely generated group.
A group $\Gamma$ is \emph{virtually $K$} if it contains an embedded copy of $K$ as a finite index subgroup.
Then $\Gamma$ is also finitely generated, so $\Gamma=\langle\Delta\rangle$ with $|\Delta|<\infty$.
When considering a virtually $K$ group $\Gamma$ as an input to an algorithm, we shall take a finite generating set $\Delta_K\subseteq\langle\Delta\rangle$ for $K$ as part of this input.

In this section we give a general theorem for proving decidability of the (kernel-based) PCP for virtually $K$ groups, which we require for Theorem \ref{thm:nilpotent}.
This method is based on the \emph{non-homogeneous PCP} (NPCP). In agreement with how we treat PCP in the rest of the paper, NPCP will ask for nontrivial solutions $x\neq 1$.
Our main result on the NPCP is \Cref{thm:virtually}; our use of the NPCP here is entirely functional and the assumption that $x\neq1$ is required for the proof.
The same problem, but allowing solutions $x=1$, was defined and studied by Myasnikov, Nikolaev and Ushakov \cite{Myasnikov2014Post}, and we denote their version by NPCP$^1$.
In the context of \Cref{thm:virtually}, NPCP is essentially equivalent to NPCP$^1$, as Lemma \ref{lem:NPCP} shows.

\p{The NPCP}
An instance of the NPCP is a tuple $I=(\Sigma, \Gamma, g, h, u_1, u_2, v_1, v_2)$ with $g, h\colon F(\Sigma)\rightarrow \Gamma$ group homomorphisms and $u_1, u_2, v_1, v_2\in \Gamma$.
The NPCP is the decision problem:
\begin{align*}
	&\textnormal{
		Given $I_{\NPCP}=(\Sigma, \Gamma, g, h, u_1, u_2, v_1, v_2)$,}\\
	&\textnormal{is there $x\in F(\Sigma)\setminus\{1\}$ such that $u_1g(x)u_2=v_1h(x)v_2$?}
\end{align*}

The identity above can be rewritten in any group as $ug(x)=h(x)v$ by letting $u=v_1^{-1}u_1$ and $v=v_2u_2^{-1}$, but we keep the statement in terms of $u_i$ and $v_i$ as this definition corresponds to what is called the ``generalised PCP'' (GPCP) for free monoids. However, we
use the ``non-homogeneous'' phrasing instead, as for our definition of the (kernel-based) PCP for groups this is a generalisation of the PCP only when one of $g$ or $h$ is injective (under this constraint, one takes $u_1=u_2=v_1=v_2=1$ to get the PCP for $\Gamma$).

By Lemma~\ref{lem:WP}, the PCP generalises the word problem.
Similarly, the NPCP generalises the conjugacy problem.
It also generalises the ``twisted conjugacy problem for pairs of endomorphisms'' \cite[Proposition 3.2]{Myasnikov2014Post}.
\begin{lemma}\label{lem:CP}
If a recursively presented group has decidable NPCP, then it has decidable
conjugacy problem.
\end{lemma}

\begin{proof}
  Let \(\Gamma\) be a finitely presented group. By the
  universal property of free groups, there exists a set $\Sigma$ with
  corresponding free group $F(\Sigma)$, and a surjective homomorphism $g\colon
  F(\Sigma)\rightarrow \Gamma$. Then let $g=h$, and observe that the instance
\[
(\Sigma, \Gamma, g, g, u, 1, 1, v)
\]
of the NPCP has a solution if and only if $u, v\in\Gamma$ are conjugate.
\end{proof}

We end our introduction to the NPCP by connecting it to the NPCP$^1$, as defined by Myasnikov, Nikolaev and Ushakov.
This lemma means that in the statement of Proposition \ref{thm:virtually}, either of NPCP or NPCP$^1$ can be used.

\begin{lemma} \label{lem:NPCP}
Let $\Gamma$ be a recursively presented group and suppose the (kernel-based) PCP is decidable for $\Gamma$. Then the NPCP is decidable for $\Gamma$ if and only if NPCP$^1$ is decidable for $\Gamma$.
\end{lemma}

\begin{proof}
Note that the word problem for $\Gamma$ is decidable by Lemma \ref{lem:WP}, since $\Gamma$ has decidable PCP.
Note also that an instance of the NPCP is also an instance of the NPCP$^1$, and vice versa. In the following, we fix an ``instance'' and write $I$ for it when viewing it as an instance of the NPCP, and $I^1$ when viewing it as an instance of the NPCP$^1$.

Suppose the NPCP is decidable for $\Gamma$.
Let $I^1$ be an instance of the NPCP$^1$; by assumption we may solve the associated instance $I$ of the NPCP.
If $I$ has a solution $x$, then this will also be a solution to $I^1$.
If $I$ has no solutions, then $I^1$ has a solution if and only if $x=1$ is a solution to $I^1$;
this is equivalent to the identity $u_1u_2=v_1v_2$ holding in $\Gamma$, which we can check as $\Gamma$ has decidable word problem.
Hence, we may solve $I^1$ as required.

Now suppose the NPCP$^1$ is decidable for $\Gamma$.
Consider an instance $I=(\Sigma, \Gamma, g, h, u_1, u_2, v_1, v_2)$ of the NPCP; by assumption we may solve the associated instance $I^1$ of the NPCP$^1$.
Firstly, assume $u_1u_2\neq v_1v_2$ in $\Gamma$; then $x=1$ cannot be a solution, so $I^1$ and $I$ have the same solution sets, and so we can solve $I$ as we can solve $I^1$.
Secondly, assume $u_1u_2 = v_1v_2$ in $\Gamma$; then $I^1$ will have solution $x=1$ but might not return any other solution.
We proceed by determining if $g,h$ are injective (standard algorithms do this for us, e.g. Nielsen reduction or Stallings' folding algorithm).
If both $g, h$ are non-injective then any $x\in \ker(g)\cap\ker(h)$ is also a solution, and moreover $\{1\}\neq \ker(g)\cap\ker(h)$
(this is clear if $\langle \ker(g), \ker(h)\rangle$ is cyclic, while otherwise $\{1\}\neq[\ker(g), \ker(h)]\leq(\ker(g)\cap\ker(h))$)
so there exists a non-trivial solution; hence we can solve $I$.
If one of $g, h$ is injective, define the word $w:=v_1^{-1}u_1=v_2u_2^{-1}$, and the map $f\colon F(\Sigma)\to \Gamma$ by $f(x)=wg(x)w^{-1}$ for all $x\in F(\Sigma)$.
This gives an instance $J=(\Sigma, \Gamma, f, h)$ of the PCP for $\Gamma$ and, as $g$ or $h$ is injective, $x\in F(\Sigma)$ is a solution to $J$ if and only if it is a solution to $I$.
By hypothesis we can solve $J$, so we can solve $I$ as required.
\end{proof}

\p{Finite index overgroups}
We now apply the NPCP to the problem of solving the PCP in finite index overgroups (Proposition \ref{thm:virtually}).
Our proof uses coset enumeration; this classical algorithm takes as input a finite presentation of a group $\Gamma$ and a finite generating set of a finite index subgroup $K$ of $\Gamma$, and outputs a representative for each left (say) coset of $K$ in $\Gamma$ \cite[Chapter 5]{Holt2005Handbook}.
The group $\Gamma$ acts on these cosets by left multiplication, and an element $x\in\Gamma$ is contained in $K$ if and only if it fixes the trivial coset $1K$.
This permutation representation therefore gives a solution to the membership problem for such a subgroup $K$ (as defined in \Cref{sec:hyperbolic}).

\begin{proposition}
	\label{thm:virtually}
	Let $K$ be a finitely presented group.
	If the (kernel-based) PCP and the NPCP are decidable for $K$, then the (kernel-based) PCP is decidable for virtually $K$ groups.
\end{proposition}

\begin{proof}
Let $\Gamma$ be a group that contains $K$ as a finite index subgroup and
	let $I=(\Sigma, \Gamma, g, h)$ be an instance of the PCP for $\Gamma$.

	As $K$ has finite index in $\Gamma$, there exists a finite index
	subgroup $N_K$ of $F(\Sigma)$ such that $g(N_K)\leq K$ and $h(N_K)\leq K$ (for example, take $N_K$ to be $g^{-1}(K)\cap h^{-1}(K)$).
As explained in the preamble, as $K$ has finite index in $\Gamma$, we can algorithmically determine membership in $K$, and so by enumerating all finite index subgroups of $F(\Sigma)$ and then computing a basis for each, we can compute a finite basis $\Sigma_K$ for some such subgroup $N_K$ (this basis is used implicitly throughout the following) as well as a set of coset representatives $p_1, \ldots, p_n$ for $N_K$ in $F(\Sigma)$ (via coset enumeration).

	We therefore have an instance $I'=(\Sigma_K, K, g|_{N_K}, h|_{N_K})$ of the PCP for $K$.
	By assumption, this problem is decidable. Moreover, its solutions correspond precisely to solutions to $I$ that are contained in $N_K$, as supposing $x\in N_K$, then $g|_{N_K}(x)=h|_{N_K}(x)$ if and only if $g(x)=h(x)$, and $x\in\ker(g|_{N_K})\cap\ker(h|_{N_K})$ if and only if $x\in\ker(g)\cap\ker(h)$.
	Therefore, we can determine if there exists some solution $x\in N_K$ to $I$. Our next step is therefore to run the PCP algorithm for $I'$, and if such a solution exists then we output that $I$ has a solution for $\Gamma$.

	Now assume that there are no solutions to $I$ belonging to $N_K$.
	We consider \emph{potential solutions} $x\in\eq(g, h)\setminus N_K$; every potential solution $x$
decomposes as $p_iq$, where $p_i\not\in N_K$, $i\in\{1, \ldots, n\}$, is one of the pre-computed coset representatives of $N_K$ and $q \in N_K$, and so
every potential solution $x$ is of one of two forms:
\begin{enumerate}
\item $x=p_i$ for some $i$ (so $x$ belongs to a finite set of known constants), or
\item $x=p_iq$ for some $i$ and some $q\in N_K\setminus\{1\}$.
\end{enumerate}
Consider the potential solutions of type (1).
As $K$ has decidable NPCP, it has decidable word problem.
As we can determine membership of $K$ for elements of $\Gamma$, the group $\Gamma$ also has decidable word problem.
Hence, for each $p_i$ we can determine if $p_i\in\eq(g, h)$.
Therefore, for each $p_i$ we can determine if there exists a potential solution of type (1).

Consider the potential solutions of type (2).
If $p_iN_K$ contains a potential solution $x=p_iq$ then $h(p_i)^{-1}g(p_i)=h(q)g(q)^{-1}\in K$, and so $I_{\NPCP}^{(i)}:=(\Sigma_K, K, g, h, h(p_i)^{-1}g(p_i), 1, 1, 1)$ is an instance of the NPCP for $K$ (as the constants are in $K$).
Then, by rearranging $g(p_i)g(q)=h(p_i)h(q)$ and noting that $q\in N_K\setminus\{1\}$, we see that $x=p_iq$ is a potential solution of type (2) if and only if $q$ is a solution to $I_{\NPCP}^{(i)}$.
By assumption, we can determine if the instance $I_{\NPCP}^{(i)}$ of the NPCP has a solution.
Therefore, for each $p_i$ we can determine if there exists a potential solution of type (2), first by determining if $h(p_i)^{-1}g(p_i)\in K$, and then by determining if $I_{\NPCP}^{(i)}$ has a solution.

Now, suppose $x_i, x_i'\in\eq(g, h)\cap p_iN_K$ are potential solutions corresponding to the same $p_i$.
Then $x_i'$ is a solution to $I$ if and only if $x_i$ is a solution to $I$, and to see this note that them being potential solutions gives us that $g(x_i)=h(x_i)$, $g(x_i')=h(x_i')$, and $x_i^{-1}x_i'=q\in N_K$.
These give us that $g(q)=h(q)$, and, by applying the assumption $I$ has no solutions in $N_K$, we have that $q\in\ker(g)\cap\ker(h)$.
Therefore, $x_i'$ is a solution to $I$ if and only if $x_i'=x_iq\not\in\ker(g)\cap\ker(h)$, if and only if $x_i\not\in\ker(g)\cap\ker(h)$, if and only if $x_i$ is a solution to $I$, as required.

Hence, for fixed $i$ we can determine if there exists some potential solution $x\in\eq(g, h)\cap p_iN_K$, and moreover we can find such an element (by first checking if $p_i$ is such an element, and if not then by enumerating the non-trivial elements of $N_K$ and then using the word problem for $K$ to check if each given element is a solution to $I_{\NPCP}^{(i)}$).
Our algorithm therefore proceeds by looping through the $p_i$, and finding an element $x_i\in\eq(g, h)\cap p_iN_K$, if one exists, as discussed above.
These finitely-many potential solutions (at most one for each $p_i$) are stored in a list $L_{\mathcal{PS}}$.
Crucially, by the above paragraph, $I$ contains a solution if and only if $L_{\mathcal{PS}}$ contains a solution.

	We now loop through the list $L_{\mathcal{PS}}$.
	So, let $x_i$ be a potential solution in $L_{\mathcal{PS}}$.
	As the word problem for $K$ is decidable, the word problem for $\Gamma$ is decidable, so we can determine if $x_i\in\ker(g)\cap\ker(h)$.
	If $x_i\not\in\ker(g)\cap\ker(h)$, then $x_i$ is a solution to $I$, so output that $I$ has a solution and terminate the algorithm.
	If $x_i\in\ker(g)\cap\ker(h)$, then, by the above paragraph, $\eq(g, h)\cap p_iN_K$ contains no solutions to $I$.
	Therefore, move on to the next $x_i$ in the list.
	If the loop ends with no solution being detected, $I$ has no solutions and so output this fact and terminate the algorithm.
\end{proof}


\section{Virtually nilpotent groups and virtually $\mathfrak{A}$ groups}
\label{sec:nilpotent}
In this section we prove Theorem \ref{thm:nilpotent}, on virtually nilpotent groups.
Like hyperbolic groups, finitely generated nilpotent groups have decidable word and conjugacy problems \cite{blackburn_nilpotent_conjugacy}.
However, most similarities with hyperbolic groups end here.
Unlike with hyperbolic groups, the satisfiability of systems of equations in (free) nilpotent groups is undecidable \cite{romankov_undecidable_first}, and
  many papers have followed this discussing different types of equations in
  various nilpotent groups \cites{duchin_liang_shapiro, Romankov_commutators,
  repin83, random_nilpotent}.
In addition, the knapsack problem is undecidable \cite{Konig2016Knapsack}.
On the other hand, and again contrasting with hyperbolic groups, the subgroup membership problem is decidable, and there exists an algorithm to compute generating sets for intersections of finitely generated subgroups \cites{Lo1998finding, baumslag_etal_polycyclic_algorithms} (many similar positive results extend to polycyclic groups).

Therefore, from an algorithmic viewpoint, hyperbolicity and nilpotency are somewhat opposite. Since the PCP is undecidable for hyperbolic groups, this intuition suggests it is decidable for nilpotent groups.
  This is indeed the case, with Myasnikov, Nikolaev and Ushakov proving decidability \cite[Theorem 5.8]{Myasnikov2014Post} (their proof actually addresses our definition of the PCP, rather than theirs, which is a mistake in their exposition).
  The purpose of this section is to extend Myasnikov, Nikolaev and Ushakov's result to virtually nilpotent groups (i.e. groups containing a nilpotent subgroup of finite index).

\p{Varieties of groups}
A \emph{variety of groups} is a class of groups $\mathfrak{A}$ closed under taking subgroups, homomorphic images and unrestricted direct products.
Equivalently, a class of groups $\mathfrak{A}$ is a variety if for all free groups $F$ there exists a subset $\mathfrak{w}_{\mathfrak{A}}\subset F$ such that the elements of $\mathfrak{A}$ which are homomorphic images of $F$ are precisely the groups $G$ such that for every homomorphism $\phi\colon F\to G$, we have $\mathfrak{w}_{\mathfrak{A}}\subseteq\ker(\phi)$. We call the set of words $\mathfrak{w}_{\mathfrak{A}}$ the \emph{laws} (or \emph{identities}) for $\mathfrak{A}$.
A variety is \emph{proper} if it is not simply the class of all groups.

Given a group $\Gamma$, we can define its ``laws'' as follows. For some infinite set of variables $\Sigma_{\infty}=\{X_1, X_2, \ldots\}$ and a finite word $w=w(X_1, \dots, X_k)\in F(\Sigma_{\infty})$ on these variables, denote by $\Gamma_w$ the set of all the \emph{values} $w$ takes in $\Gamma$; that is, $\Gamma_w=\{w(g_1, \dots, g_k) \mid g_i \in \Gamma\}$.
A word $w=w(X_1, \dots, X_k)$ as above is a \emph{law} (or \emph{identity}) in $\Gamma$ if
$\Gamma_w=\{1\}$, or equivalently, $w(g_1, \dots, g_k)=1$ in $\Gamma$ for any choice of $g_1, \dots, g_k \in \Gamma$.
A variety $\mathfrak{A}$ is then simply the class of all those groups $\Gamma$ for which $\mathfrak{w}_{\mathfrak{A}}$ are laws in $\Gamma$.

For example, the commutator $[X_1, X_2]$ is a law in any abelian group, and higher commutators are laws in the appropriate class nilpotent groups.
Thus the classes of all groups, abelian groups, nilpotent groups (of arbitrary class or of fixed class $c$), soluble groups (of arbitrary derived length or of fixed derived length $d$), and periodic groups (of arbitrary exponent or fixed exponent $e$) each form a variety.

Varieties have free objects: If $F(\Sigma)$ is a free group, we write $\langle\mathfrak{w}_{\mathfrak{A}}\rangle$ for the minimal normal subgroup of $F(\Sigma)$ such that $F(\Sigma)/\langle\mathfrak{w}_{\mathfrak{A}}\rangle$ is in $\mathfrak{A}$, and so the quotient $F_{\mathfrak{A}}(\Sigma):=F(\Sigma)/\langle\mathfrak{w}_{\mathfrak{A}}\rangle$ is free of rank $|\Sigma|$ in this variety.

Most of our proof of Theorem \ref{thm:nilpotent} is in the setting of the very general world of varieties of groups; as nilpotent groups form a variety, the results we prove are immediately applicable to nilpotent groups.
We refer the reader to H. Neumann's classic text for background and definitions on varieties of groups \cite{Neumann1967varieties}, but a reader interested in just nilpotent groups can simply take $\mathfrak{A}$ in the following to be the class of nilpotent groups.

\p{Varieties of recursively/finitely presented groups}
The PCP is only defined for finitely generated recursively presented groups, while non-trivial varieties necessarily contain non-finitely generated groups.
The following definitions are therefore needed to discuss the PCP for varieties of groups.
By the \emph{weak variety of recursively (finitely) presented groups} $\mathfrak{A}$ we mean all the recursively (finitely) presented groups in a given variety $\mathfrak{B}$; we further call $\mathfrak{A}$ a \emph{variety of recursively (finitely) presented groups}, i.e. we drop the ``weak'', if for all $|\Sigma|<\infty$ the free objects $F_{\mathfrak{B}}(\Sigma)$ are contained in $\mathfrak{A}$.
In this case, we associate both the laws of $\mathfrak{A}$ and $\mathfrak{B}$, i.e. $\mathfrak{w_A}:=\mathfrak{w_B}$, and the free objects of $\mathfrak{A}$ and $\mathfrak{B}$,
i.e. $F_{\mathfrak{A}}(\Sigma):=F_{\mathfrak{B}}(\Sigma)$.

We consider {virtually $\mathfrak{A}$} groups, for $\mathfrak{A}$ a variety of recursively presented groups, and prove that if we can solve two specific algorithmic problems for $\mathfrak{A}$ groups then we can solve the PCP for virtually $\mathfrak{A}$ groups.
In particular, we can solve these problems in nilpotent groups, so Theorem \ref{thm:nilpotent} follows.

Like in Section \ref{sec:virtually}, when considering a virtually $\mathfrak{A}$ group $\Gamma$ as an input to an algorithm, we shall take a finite generating set $\Delta_K\subseteq\langle\Delta\rangle$ for the finite index subgroup $K\in\mathfrak{A}$ as part of this input.

  For $\mathfrak{A}$ a variety of recursively presented groups, the \emph{PCP for $\mathfrak{A}$}, written $\mathfrak{A}$-PCP, is the PCP restricted to those instances $I=(\Sigma, \Gamma, g, h)$ where the group $\Gamma$ is in $\mathfrak{A}$.
We define the non-homogeneous PCP (see Section \ref{sec:virtually}) \emph{NPCP for $\mathfrak{A}$}, written $\mathfrak{A}$-NPCP, analogously.

\subsection{The NPCP for varieties of groups}
In this section we give certain conditions which imply that the NPCP is decidable for a
variety of recursively presented groups.

Our conditions which imply that the $\mathfrak{A}$-NPCP is decidable are as follows.
These conditions do not hold for the varieties of all recursively or finitely presented groups, thus we may assume the word ``proper'' in the statement.
Note also that for any instance $I$ of the $\mathfrak{A}$-PCP, the subgroup $\langle\mathfrak{w}_{\mathfrak{A}}\rangle\leq F(\Sigma)$ generated by the laws of $\mathfrak{A}$ is contained in $\eq(g, h)$ and so (\ref{NPCPvarieties:1}) may be viewed as requesting a finite generating set for the quotient $\eq(g, h)/\langle\mathfrak{w}_{\mathfrak{A}}\rangle$.

\begin{proposition}
\label{thm:NPCPvarieties}
Let $\mathfrak{A}$ be a (proper) variety of recursively or finitely presented groups.
Suppose the following hold.
\begin{enumerate}
\item\label{NPCPvarieties:1} There exists an algorithm with input an instance $I=(\Sigma, K, g, h)$ of the $\mathfrak{A}$-PCP, and output a finite set $S\subset F(\Sigma)$ such that $\eq(g, h)=\langle S,\mathfrak{w}_{\mathfrak{A}} \rangle$.
\item\label{NPCPvarieties:2} There exists an algorithm with input a group $K\in\mathfrak{A}$, a pair of finitely generated subgroups $A, B<K$, and an element $x\in K$, and which determines if $x$ is contained in the product of the subgroups $A$ and $B$, i.e. if $x\in AB$.
\end{enumerate}
Then the $\mathfrak{A}$-NPCP is decidable.
\end{proposition}

Our proof of \Cref{thm:NPCPvarieties} needs a preliminary setup.

The verbal product $VP_{\mathfrak{A}}(A, B)$ of groups $A, B\in\mathfrak{A}$, as defined by H. Neumann \cite[Definition 18.31]{Neumann1967varieties}, plays the role of a free product in the variety $\mathfrak{A}$.
In particular, the group $VP_{\mathfrak{A}}(A, B)$ is contained in $\mathfrak{A}$ if $A, B\in\mathfrak{A}$, and both $A$ and $B$ embed into this group.
If $\mathfrak{A}$ is instead a variety of recursively/finitely presented groups, then $VP_{\mathfrak{A}}(A, B)$ is contained in $\mathfrak{A}$ if $A, B\in\mathfrak{A}$
as the verbal product is the quotient of the free product $A\ast B$ by (the subgroup generated by) the laws of $\mathfrak{A}$ and so is recursively/finitely presented.
In particular, writing $F_{\mathfrak{A}}(\alpha, \omega)$ for the free group in the variety $\mathfrak{A}$ over the alphabet $\{\alpha, \omega\}$, if $K\in\mathfrak{A}$ then $VP_{\mathfrak{A}}(K, F_{\mathfrak{A}}(\alpha, \omega))\in\mathfrak{A}$.

We start with an instance $I_{\NPCP}=(\Sigma, K, g, h, u_1, u_2, v_1, v_2)$ of the $\mathfrak{A}$-NPCP and consider the instance
\[
I_{\PCP}=(\Sigma\sqcup\{\alpha, \omega\}, VP_{\mathfrak{A}}(K, F_{\mathfrak{A}}(\alpha, \omega)), g', h')
\]
of the $\mathfrak{A}$-PCP, where $g'$ and $h'$ are defined as follows.
\begin{align*}
g'(z):=
\begin{cases}
g(z)&\textrm{if} \ z\in\Sigma\\
\alpha u_1&\textrm{if} \ z=\alpha\\
u_2\omega&\textrm{if} \ z=\omega
\end{cases}
&&
h'(z):=
\begin{cases}
h(z)&\textrm{if} \ z\in\Sigma\\
\alpha v_1&\textrm{if} \ z=\alpha\\
v_2\omega&\textrm{if} \ z=\omega
\end{cases}
\end{align*}

We now connect the solutions of $I_{\NPCP}$ to those of $I_{\PCP}$.

\begin{lemma}
\label{lem:solnforNPCP}
A word $y\in F(\Sigma)\setminus\{1\}$ is a solution to $I_{\NPCP}$ if and only if the word $\alpha y\omega$ is a solution to $I_{\PCP}$.
\end{lemma}

\begin{proof}
Note first that the element $\alpha k\omega$ of $VP_{\mathfrak{A}}(K, F_{\mathfrak{A}}(\alpha, \omega))$ is non-trivial for all $k\in K$.
This is because there is a natural projection $VP_{\mathfrak{A}}(K, F_{\mathfrak{A}}(\alpha, \omega))\twoheadrightarrow F_{\mathfrak{A}}(\alpha, \omega)$ taking $\alpha k\omega$ to $\alpha\omega\neq1$ \cite[Paragraph 18.33]{Neumann1967varieties}.
It follows that $\alpha y\omega\not\in\ker(g')\cap\ker(h')$ for all $y\in F(\Sigma)$, because the images of $\alpha y\omega$ under $g'$ and $h'$ have the required form $\alpha k\omega$, $k\in K$.
Therefore, it is sufficient to prove that $y$ is a solution to $I_{\NPCP}$ if and only if $\alpha y\omega\in\eq(g', h')$.

Starting with $y$ being a solution to $I_{\NPCP}$, we obtain the following sequence of equivalent identities:
\begin{align*}
u_1g(y)u_2&=v_1h(y)v_2\\
\alpha u_1 g(y) u_2\omega&=\alpha v_1 h(y) v_2\omega\\
g'(\alpha)g'(y)g'(\omega)&=h'(\alpha)h'(y)h'(\omega)\\
g'(\alpha y\omega)&=h'(\alpha y\omega)
\end{align*}
Therefore $\alpha y\omega$ is a solution to $I_{\PCP}$, so the claimed equivalence follows.
\end{proof}

We now prove \Cref{thm:NPCPvarieties}.
Recalling that $\langle\mathfrak{w}_{\mathfrak{A}}\rangle$ denotes the minimal normal subgroup of $F(\Sigma)$ such that $F_{\mathfrak{A}}(\Sigma)=F(\Sigma)/\langle\mathfrak{w}_{\mathfrak{A}}\rangle$ is in $\mathfrak{A}$,
we write $\pi_{(\Sigma, \mathfrak{A})}: F(\Sigma)\to F_{\mathfrak{A}}(\Sigma)$ for the associated homomorphism.
For an instance $I=(\Sigma', K', g', h')$ of the $\mathfrak{A}$-PCP, we have $\langle\mathfrak{w}_{\mathfrak{A}}\rangle\leq\ker(g')\cap\ker(h')$ and so there exist maps $g'_{\mathfrak{A}}, h'_{\mathfrak{A}}: F_{\mathfrak{A}}(\Sigma')\to K'$ such that $g'$ and $h'$ decompose as $g'_{\mathfrak{A}}\pi_{(\Sigma', \mathfrak{A})}$ and $h'_{\mathfrak{A}}\pi_{(\Sigma', \mathfrak{A})}$, respectively.
In the following proof, we consider the equaliser $\eq(g'_{\mathfrak{A}}, h'_{\mathfrak{A}})$, which is equal to the quotient $\eq(g', h')/\langle\mathfrak{w}_{\mathfrak{A}}\rangle$.

\begin{proof}[Proof of \Cref{thm:NPCPvarieties}]
Let $I_{\NPCP}=(\Sigma, K, g, h, u_1, u_2, v_1, v_2)$ be a given instance of the $\mathfrak{A}$-NPCP.
The algorithm begins by computing the instance $I_{\PCP}$ as in the preliminary setup.
By construction, $I_{\PCP}$ is an instance of the $\mathfrak{A}$-PCP, and so next apply the algorithm of (\ref{NPCPvarieties:1}) to compute a finite subset $S$ of $F(\Sigma')$, where $\Sigma'=\Sigma\sqcup\{\alpha, \omega\}$, such that $\eq(g', h')=\langle S,\mathfrak{w}_{\mathfrak{A}} \rangle$.
Consider the subgroups $P:=\eq(g'_{\mathfrak{A}}, h'_{\mathfrak{A}})$ and $Q:=F_{\mathfrak{A}}(\Sigma)$ of $F_{\mathfrak{A}}(\Sigma')$, which are given in terms of the explicit finite generating sets $S'=\{s\langle\mathfrak{w}_{\mathfrak{A}}\rangle \mid s\in S\}$ and $\Sigma$.
Finally, writing $P^{\omega}$ for $\omega P\omega^{-1}$, use the algorithm of (\ref{NPCPvarieties:2}) to determine if the element $\omega\alpha$ is contained in the product of subgroups $P^{\omega}Q\subseteq F_{\mathfrak{A}}(\Sigma')$.

The result now follows from the following claim: The element $\alpha\omega\in F_{\mathfrak{A}}(\Sigma')$ is contained in the product $P^{\omega}Q$ if and only if $I_{\NPCP}$ has a solution. To prove the claim, note that $\omega\alpha\in P^{\omega}Q$
if and only if
there exist some $p\in P, q\in Q$ such that $\omega\alpha= {\omega}p\omega^{-1}q$,
or equivalently, $\alpha q^{-1}\omega= p\in \eq(g'_{\mathfrak{A}}, h'_{\mathfrak{A}})$.
By taking $y_{\mathfrak{A}}:=q^{-1}$, this holds if and only if
there exists some $y_{\mathfrak{A}}\in F_{\mathfrak{A}}(\Sigma)$ such that $\alpha y_{\mathfrak{A}}\omega\in\eq(g'_{\mathfrak{A}}, h'_{\mathfrak{A}})$.
By taking $y\in\pi_{(\Sigma, \mathfrak{A})}^{-1}(y_{\mathfrak{A}})$ with $y\neq1$, which we may do as $\pi_{(\Sigma, \mathfrak{A})}$ is non-injective as $\mathfrak{A}$ is proper, we see that this holds if and only if there exists some $y\in F(\Sigma)\setminus\{1\}$ such that $\alpha y\omega\in\eq(g', h')$.
By Lemma \ref{lem:solnforNPCP} such a $y\in F(\Sigma)\setminus\{1\}$ exists if and only if $I_{\NPCP}$ has a solution, as claimed.
\end{proof}


\subsection{The PCP in virtually $\mathfrak{A}$ groups}
\label{sec:virtuallyA}
Applying Proposition \ref{thm:virtually} to Proposition \ref{thm:NPCPvarieties} gives the following.

\begin{theorem}
\label{thm:virtuallyA}
Let $\mathfrak{A}$ be a variety of finitely presented groups.
Suppose (\ref{NPCPvarieties:1}) and (\ref{NPCPvarieties:2}) of Proposition \ref{thm:NPCPvarieties} hold.
Then the PCP is decidable for virtually $\mathfrak{A}$ groups.
\end{theorem}

\begin{proof}
Note that, by Proposition \ref{thm:NPCPvarieties} (\ref{NPCPvarieties:2}),
groups in $\mathfrak{A}$ have decidable word problem (take $A$ and $B$ to be trivial).

Let $I=(\Sigma, K, g, h)$ be an instance of the $\mathfrak{A}$-PCP.
We start by producing a finite set $S\subset F(\Sigma)$ such that $\langle S,\mathfrak{w}_{\mathfrak{A}} \rangle=\eq(g, h)/\langle\mathfrak{w}_{\mathfrak{A}}\rangle$, which we can do by Proposition \ref{thm:NPCPvarieties} (\ref{NPCPvarieties:1}).
Then $\eq(g, h)/(\ker(g)\cap\ker(h))$ is trivial if and only if $S\subseteq \ker(g)\cap\ker(h)$; since $S$ is finite and $K$ has decidable word problem, we can determine if $S\subseteq \ker(g)\cap\ker(h)$ and
thus solve the $\mathfrak{A}$-PCP.

The assumptions of this theorem match the assumptions of \Cref{thm:NPCPvarieties}, so the $\mathfrak{A}$-NPCP is also decidable.
The result now follows from \Cref{thm:virtually}.
\end{proof}

We can now prove Theorem \ref{thm:nilpotent}.

\begin{proof}[Proof of Theorem \ref{thm:nilpotent}]
Both problems required for Theorem \ref{thm:virtuallyA} are decidable for nilpotent groups \cite[Theorem 5.7]{Myasnikov2014Post} \cite[Algorithms 6.2 and 6.1]{Lo1998finding}, so the result follows from Theorem \ref{thm:virtuallyA}, along with
the fact that finitely generated nilpotent groups are finitely presented \cite[Theorem 3.4]{baumslag_etal_polycyclic_algorithms}.
\end{proof}

\begin{remark}
It can be noted that the algorithm in Theorem~\ref{thm:virtuallyA} is uniform,
in the sense that there is one algorithm that works for all virtually $\mathfrak{A}$ groups,
modulo uniform solutions to the two problems in Proposition \ref{thm:NPCPvarieties} for $\mathfrak{A}$.
These problems are uniformly decidable for finitely generated nilpotent groups, and so
Theorem \ref{thm:nilpotent} gives solution to the uniform PCP for finitely generated nilpotent groups.
\end{remark}


\section{The verbal PCP}
\label{sec:verbal}

In this section we consider the version of the PCP defined by
Myasnikov, Nikolaev and Ushakov \cite{Myasnikov2014Post}, which we call the
\emph{verbal PCP} because it looks for solutions outside a verbal subgroup, as explained below.
To differentiate, recall that we refer to the prevalent
version of PCP in this paper (as defined in Section
\ref{sec:prelim}) as the \emph{kernel-based PCP}.

Recall from Section \ref{sec:nilpotent} that the \emph{values} of a word $w=w(X_1, \dots, X_k)\in F(\Sigma_{\infty})$ in a group $G$ is the set $G_w=\{w(g_1, \dots, g_k) \mid g_i \in G\}$.
A \emph{verbal subgroup} of $G$ is a subgroup generated by the values of a set of words $\mathfrak{w}\subset F(\Sigma_{\infty})$.
As the set $\mathfrak{w}$ is typically given, we write $\mathfrak{w}(G)$ for this subgroup, so $\mathfrak{w}(G)=\langle G_w\mid w\in\mathfrak{w} \rangle$.
Recalling that a word $w\in F(\Sigma_{\infty})$ is a law in $\Gamma$ if $\Gamma_w=\{1\}$,
we are particularly interested in the verbal subgroup $\mathfrak{w}_{\Gamma}(F(\Sigma))$, shortened to $\langle\mathfrak{w}_{\Gamma}\rangle$ when $F(\Sigma)$ is implicit, where $\mathfrak{w}_{\Gamma}$ is the set of laws of $\Gamma$.
Equivalently, $\langle\mathfrak{w}_{\Gamma}\rangle$ is the maximal normal subgroup of $F(\Sigma)$ such that every homomorphism $F(\Sigma)\to\Gamma$ factors through $F(\Sigma)/\langle\mathfrak{w}_{\Gamma}\rangle$.

The definition of the verbal PCP we give now is equivalent to  Myasnikov, Nikolaev and Ushakov's definition, but has been rephrased to mirror the definition of the kernel-based PCP:
An \emph{instance} of the verbal PCP is an instance of the kernel-based PCP, so a four-tuple $I=(\Sigma, \Gamma, g, h)$ with $g, h\colon F(\Sigma)\rightarrow \Gamma$.
The \emph{verbal PCP} itself is the decision problem: \\

\begin{center}
Given $I=(\Sigma, \Gamma, g, h)$, is the group $\eq(g, h)/\langle\mathfrak{w}_{\Gamma}\rangle$ trivial?
\end{center}

\medskip

Note that $\langle\mathfrak{w}_{\Gamma}\rangle\leq \ker(g)\cap\ker(h)$, so for a fixed instance $I$, the verbal PCP may have solutions when the kernel-based PCP does not.

\p{Hyperbolic groups}
Suppose $\Gamma$ is non-elementary hyperbolic.
Then $\Gamma$ contains a non-abelian free group, so $\langle\mathfrak{w}_{\Gamma}\rangle$ is trivial.
Hence, the verbal PCP for non-elementary hyperbolic groups is simply asking if the equaliser $\eq(g, h)$ is trivial.
This compares with the kernel-based PCP as follows:
\begin{enumerate}
\item\label{hyperbolic:1}
If either $g$ or $h$ is injective then $\ker(g)\cap\ker(h)=\langle\mathfrak{w}_{\Gamma}\rangle=\{1\}$.
Hence, the verbal PCP and the kernel-based PCP ask the same question and so have identical solution sets.

Decidability when $g$ or $h$ is injective is unknown.
\item
If both $g$ and $h$ are non-injective then $\ker(g)\cap\ker(h)$ is non-trivial (as it contains the non-trivial subgroup $[\ker(g), \ker(h)]$).
Hence, the verbal PCP necessarily has a solution, and so is trivially decidable.
However, by \Cref{thm:hyperbolic}, the kernel-based PCP is undecidable.
\end{enumerate}

\p{Nilpotent groups}
As noted in Section \ref{sec:nilpotent}, Myasnikov, Nikolaev and Ushakov
\cite{Myasnikov2014Post} proved the kernel-based PCP for nilpotent groups, rather than the verbal PCP as their theorem incorrectly states.
We now rectify this situation and prove that the verbal PCP is decidable for torsion-free nilpotent groups.

We need the following background on torsion-free nilpotent groups (see \cite{Nickel, Cant-Eick}).
Let $\Gamma$ be a finitely generated torsion-free nilpotent group. Then $\Gamma$ has a central series
\[
\Gamma=\Gamma_1 > \dots > \Gamma_{n+1}=\{1\}
\]
with $\Gamma_i/\Gamma_{i+1}$ infinite cyclic for every $1\leq i \leq n$, and one can choose $a_1, \dots, a_n \in \Gamma$ such that $\Gamma_{i-1}=\langle a_i, \Gamma_i\rangle$. Such a sequence is called a \emph{nilpotent generating sequence} for $\Gamma$, and gives rise to a \emph{nilpotent presentation}
\begin{equation}\label{eq:nilpotent}
\langle a_1, \dots, a_n \mid [a_i,a_j]=a^{c_{i,j,j+1}}_{j+1}\dots a^{c_{i,j,n}}_n, 1\leq i < j\leq n \rangle,
\end{equation}
where $c_{i,j,k} \in \mathbb{Z}$. This presentation is called \emph{consistent} in the literature (see \cite[Chapter 8]{Holt2005Handbook}) and in particular every element in $\Gamma$ can be written \emph{uniquely} as $a_1^{x_1} \dots a_n^{x_n}$, $x_i \in \mathbb{Z}$. If we write $x= (x_1, \dots, x_n)$ and $y=(y_1, \dots, y_n)$, the multiplication of two elements can be expressed as
\begin{equation}\label{eq:mult}
(a_1^{x_1} \dots a_n^{x_n})(a_1^{y_1} \dots a_n^{y_n})=a_1^{\delta_1(x, y)} \dots a_n^{\delta_n(x, y)},
\end{equation}
where by a classical result of Hall the functions $\delta_i\colon\mathbb{Z}^n\oplus \mathbb{Z}^n \mapsto \mathbb{Z}$ are rational polynomials depending on the group only, and not on the elements involved. The polynomials $\delta_i$ on $2n$ variables are called \emph{Hall polynomials} in the literature and can easily be extended to capture the multiplication of a fixed number (not just two) of elements.

Recall that we denote by $\langle\mathfrak{w}_{\Gamma}\rangle$ the verbal subgroup of $F(\Sigma)$ generated by the set $\mathfrak{w}_{\Gamma}$ of laws of $\Gamma$.

\begin{lemma}
	\label{lem:CheckLaw}
Let $F(\Sigma)$ be a free group as above, let $w\in F(\Sigma)$, and let $\Gamma$ be a finitely generated, torsion-free, nilpotent group. Then there is an algorithm that can determine whether $w\in \langle\mathfrak{w}_{\Gamma}\rangle$, that is, whether $w$ is a law of $\Gamma$.
\end{lemma}

\begin{proof}
Suppose $\Sigma=\{X_1, \dots, X_k\}$, with $k\geq 1$, and write
\[
w=w(X_1, \dots, X_k)=X_{i_1}^{E_1} \dots X_{i_{\ell}}^{E_{\ell}},
\]
where $ i_j \in \{1, \dots, {\ell}\}$ for $1\leq j \leq {\ell}$, and $E_j \in \mathbb{Z}$. Let $\{a_1, \dots, a_n\}$ be the nilpotent generating sequence of $\Gamma$ as in (\ref{eq:nilpotent}). Then, writing $w(\Gamma)$ for the minimal verbal subgroup of $\Gamma$ containing $w$, we have
\[
w(\Gamma)=\{(a_1^{e_{1,1}} \dots a_n^{e_{1,n}})^{E_1} \dots (a_1^{e_{{\ell},1}} \dots a_n^{e_{{\ell},n}})^{E_{\ell}} \mid e_{i,j} \in \mathbb{Z}\}.
\]
Since the Hall polynomials can be extended to capture the multiplication of any fixed number of elements, and the $E_i$ are fixed, there exist rational polynomials $P_1, \dots, P_n\colon \mathbb{Z}^{n{\ell}} \mapsto \mathbb{Z}$ to give the exponents of the nilpotent generating sequence. Writing $e_j=(e_{j,1}, \dots, e_{j,n})$, one gets
\[
w(\Gamma)=\{a_1^{P_1(e_1, \dots, e_n)} \dots a_n^{P_n(e_1, \dots, e_n)} \mid e_j \in \mathbb{Z}^n\}.
\]
Since the Hall polynomials for the product of two elements can be explicitly, algorithmically, computed, the polynomials $P_i$ for the multiplication within $w(\Gamma)$ can also be explicitly determined by induction.

To see that $w$ is a law of $\Gamma$ it suffices to check that $w(\Gamma)=\{1\}$. As expressing group elements in $\Gamma$ over the nilpotent generators $a_i$ is done uniquely, $w(\Gamma)=\{1\}$ if and only if the $P_i$ are all equal to the zero polynomial. This is easily checked as the Hall polynomials for $w(\Gamma)$ were explicitly computed.
\end{proof}

This allows us to solve the verbal PCP for torsion-free nilpotent groups.

\begin{theorem}
\label{thm:correction}
The verbal PCP is decidable for torsion-free nilpotent groups.
\end{theorem}

\begin{proof}
Let $I=(\Sigma, \Gamma, g, h)$ be an instance of the verbal PCP, let $\Gamma$ torsion-free nilpotent of class $c$ (that is, $\gamma_{c+1}(\Gamma)=\{1\}$), and let $\gamma_i(F(\Sigma))$ be the groups in the lower central series of $F(\Sigma)$.

By \cite[Theorem 5.7 (1)]{Myasnikov2014Post} the equaliser $\eq(g, h)$ contains $\gamma_{c+1}(F(\Sigma))$ (this can be seen by noting that $\gamma_{c+1}(F(\Sigma)) \leq \ker(g) \cap \ker(h) \leq \eq(g, h)$) and is finitely generated modulo $\gamma_{c+1}(F(\Sigma))$. That is, there exists a finite set $S\subset F(\Sigma)$ such that $\langle S,\gamma_{c+1}(F(\Sigma)) \rangle=\eq(g, h)$.
Moreover, by \cite[Theorem 5.7 (2)]{Myasnikov2014Post} one can algorithmically find $S\subset F(\Sigma)$.

Then $\eq(g, h)/\langle\mathfrak{w}_{\Gamma}\rangle$ is trivial if and only if $S\subseteq \langle\mathfrak{w}_{\Gamma}\rangle$.
Given any $x\in F(\Sigma)$, we can check if $x\in \langle\mathfrak{w}_{\Gamma}\rangle$ by Lemma \ref{lem:CheckLaw}.
Hence, as $S$ is finite, we can determine if $S\subseteq \langle\mathfrak{w}_{\Gamma}\rangle$.
Therefore, we can solve the verbal PCP for $\Gamma$.
\end{proof}

Extending \Cref{thm:correction} to general nilpotent groups would require an extension of \Cref{lem:CheckLaw} to general nilpotent groups, which seems possible.
However, extending \Cref{thm:correction} to virtually nilpotent groups, as in \Cref{thm:nilpotent}, seems difficult, essentially because the varieties change with the groups involved.
For example, in the proof of \Cref{thm:virtually}, given $I=(\Sigma, \Gamma, g, h)$ we use the finite index subgroup $K<\Gamma$ to construct a new instance $I'$ of the kernel-based PCP such that ``it's solutions correspond precisely to the solutions to $I$ which are contained in $N_K$'', $N_K$ a finite index subgroup of $F(\Sigma)$.
This holds because $N_K/(\ker(g)\cap\ker(h))$ embeds into $F(\Sigma)/(\ker(g)\cap\ker(h))$.
For the verbal PCP we would require $N_K/\mathfrak{w}_{K}(N_K)$ to embed into $F(\Sigma)/\mathfrak{w}_{\Gamma}(F(\Sigma))$, which does not happen in general.


\bibliographystyle{amsalpha}
\bibliography{BibTexBibliography}

\end{document}